\documentclass[times,review]{elsarticle}
\journal{Journal of Computational Physics}
\usepackage{framed,multirow}

\usepackage[utf8]{inputenc}

\usepackage{graphicx}
\usepackage[version=4]{mhchem}
\usepackage{longtable,tabularx}
\usepackage{amsmath,amssymb,amsthm}

\usepackage{csquotes}

\usepackage{geometry}
\usepackage{float}

\usepackage{url}
\usepackage{pifont}
\usepackage{empheq}
\usepackage{bm}
\usepackage{soul}

\usepackage{scalerel}

\usepackage{amsthm}

\usepackage{csquotes}

\usepackage{booktabs}
\usepackage{geometry}
\usepackage{float}

\usepackage{url}
\usepackage{pifont}

\usepackage{empheq}
\usepackage{bm}
\usepackage[english]{babel} 
\addto\captionsenglish{\renewcommand{\appendixname}{Appendix }}
\usepackage{amsmath,amssymb}
\usepackage{epstopdf}
\usepackage{relsize}
\usepackage{stmaryrd} 
\usepackage{subcaption}
\usepackage{multirow}
\usepackage{array,makecell}
\usepackage[squaren, Gray, cdot]{SIunits}

\usepackage{import}
\usepackage{diagbox}

\usepackage{csquotes}

\usepackage{xspace}

\usepackage{scalerel}

\addto\captionsenglish{\renewcommand{\appendixname}{}}

\makeatletter
\def\ps@pprintTitle{%
 \let\@oddhead\@empty
 \let\@evenhead\@empty
 \def\@oddfoot{}%
 \let\@evenfoot\@oddfoot}

\usepackage{hyperref}

\newtheorem{thm}{Theorem}[section]

\usepackage{xcolor}

\usepackage{algorithm}
\usepackage{algpseudocode}

\newtheorem{remark}{Remark}

\newcommand{\fnc}[1]{\ensuremath{\mathcal{#1}}}

\newcommand{\xil}[0]{\ensuremath{\xi_{l}}}

\newcommand{\matAlmtwol}[1]{\ensuremath{\left[\fnc{J}\frac{\partial\xil}{\partial x_{m}}\right]_{2l}}}
\newcommand{\matAlmtwolmone}[1]{\ensuremath{\left[\fnc{J}\frac{\partial\xil}{\partial x_{m}}\right]_{(2l-1)}}}
\newcommand{\matAlmtwoltilde}[1]{\ensuremath{\widetilde{\left[\fnc{J}\frac{\partial\xil}{\partial x_{m}}\right]}_{2l}}}
\newcommand{\matAlmtwolmonetilde}[1]{\ensuremath{\widetilde{\left[\fnc{J}\frac{\partial\xil}{\partial x_{m}}\right]}_{(2l-1)}}}

\makeatletter
\def\ps@pprintTitle{%
 \let\@oddhead\@empty
 \let\@evenhead\@empty
 \def\@oddfoot{}%
 \let\@evenfoot\@oddfoot}

\definecolor{newcolor}{rgb}{.8,.349,.1}

\usepackage{listings}
\definecolor{dkgreen}{rgb}{0,0.6,0}
\definecolor{gray}{rgb}{0.5,0.5,0.5}
\definecolor{mauve}{rgb}{0.58,0,0.82}

\begin{document}


\begin{frontmatter}

\title{
Scalable Evaluation of Hadamard Products with Tensor Product Basis for Entropy-Stable High-Order Methods}

\author[1]{Alexander {Cicchino}\corref{cor1}\fnref{CICC}}
\cortext[cor1]{Corresponding author. 
}
  \ead{alexander.cicchino@mail.mcgill.ca}
  \fntext[CICC]{Ph.D. Student}

\author[1]{Siva {Nadarajah}\fnref{Nadarajah}}
\fntext[Nadarajah]{Professor}
\ead{siva.nadarajah@mcgill.ca}

\address[1]{Department of Mechanical Engineering, McGill University, 
Montreal, QC, H3A 0C3, Canada}






\begin{abstract}

\end{abstract}
\end{frontmatter}

\section{Introduction}

Sum-Factorization techniques were introduced by Orzag~\cite{orszag1979spectral} to efficiently evaluate spectral methods. Orzag~\cite{orszag1979spectral} made use of the tensor product nature of the basis functions to perform the operations in each direction independently, and result in $\mathcal{O}\left(n^{d+1} \right)$ flops for interpolation, projection and differentiation operations. Unfortunately, a tensor product algorithm resulting in $\mathcal{O}\left(n^{d+1}\right)$ flops for Hadamard products does not yet exist in the literature. The aim of this technical note is to demonstrate that Hadamard products can be computed in $\mathcal{O}\left(n^{d+1}\right)$ flops with a tensor product basis, provided the basis functions have one additional property that is common in the spectral and finite element communities.

Entropy stable numerical schemes, initially proposed by Tadmor~\cite{tadmor1984skew} for finite-volume methods, guarantee robustness on extremely coarse meshes. Through the application of summation-by-parts (SBP) operators, and introducing flux differencing techniques, Fisher~\textit{et al}.~\cite{fisher2013high,fisher2013discretely} made the concepts from Tadmor applicable in a finite-element framework. This led to the development of provably nonlinearly stable high-order methods~\cite{fisher2013high,fisher2013discretely,fisher2012high,carpenter2014entropy,parsani2015entropy,parsani2015entropyWall,parsani2016entropy,carpenter2016towards,yamaleev2017family,crean2018entropy,chen2017entropy,Crean2019Staggered,FriedrichEntropy2020}, in a collocated split-form discontinuous Galerkin (DG) form~\cite{gassner2013skew,gassner2016split}, collocated split-form flux reconstruction (FR) framework recovering the DG case~\cite{ranocha2016summation,ranocha2017extended,abe2018stable}, modal, uncollocated entropy stable DG framework~\cite{chan2018discretely,chan2019skew,chan2019discretely,chan2019efficient}, and modal, uncollocated nonlinearly stable FR (NSFR) schemes~\cite{CicchinoNonlinearlyStableFluxReconstruction2021,cicchino2022provably}.

In the application of flux differencing~\cite[Eq. (3.9)]{fisher2013high}, Ranocha~\textit{et al}.~\cite{ranocha2021efficient} numerically demonstrated that $\mathcal{O}\left(n^{d+1}\right)$ flops could be recovered. For modal, uncollocated schemes, the general expression requires the computation of a dense Hadamard product, as seen in Chan~\cite[Eq. (58)]{chan2018discretely}. The focus of this short note is on efficiently evaluating a Hadamard product using a tensor product basis. Specifically, using the tensor product structure, we demonstrate that a Hadamard product can be assembled and evaluated in $\mathcal{O}\left(n^{d+1}\right)$ flops and memory allocation, rather than $\mathcal{O}\left(n^{2d}\right)$, where $d$ is the dimension. We term the algorithm a \enquote{sum-factorized} Hadamard product because we recover the scaling result of sum-factorization techniques~\cite{orszag1979spectral} by exploiting the tensor-product structure in the Hadamard product. This result is dependent on the basis operators being diagonal operators in at least $d-1$ directions---fortunately this is always the case for Hadamard products involving interpolation, projection and differentiation operators of polynomial basis functions. This is the case because we can use sum-factorization techniques to project onto a collocated Lagrange basis, evaluate the Hadamard product using our proposed algorithm, then project back onto the dense basis. In Section~\ref{sec: results}, we provide numerical results showing that the Hadamard product scales at $\mathcal{O}\left(n^{d+1}\right)$. Then we numerically show the application in our in-house partial differential equation solver $\texttt{PHiLiP}$~\cite{shi2021full} based on the Nonlinearly Stable Flux Reconstruction scheme~\cite{CicchinoNonlinearlyStableFluxReconstruction2021,cicchino2022provably} and demonstrate that the entire solver scales at $\mathcal{O}\left(n^{d+1}\right)$ for three-dimensional compressible flow on curvilinear grids, in a low-storage manner. Lastly, we compare the computational costs between a conservative strong form DG scheme, an over-integrated conservative strong form DG scheme, and our NSFR entropy conserving scheme. The NSFR entropy conserving scheme is the only scheme that requires a dense Hadamard product evaluation. We numerically demonstrate that with our proposed sum-factorized Hadamard products, the NSFR entropy conserving scheme is computationally competitive with the DG conservative strong form, and that over-integration schemes take significantly more computational time.

\section{Hadamard Product}\label{sec: hadamard product}

Consider solving $\left(\bm{A}\otimes \bm{B}\right)\circ \bm{C}$, with $\bm{A},\:\bm{B}\in \mathbb{R}^{n\times n}$ and $\bm{C}\in\mathbb{R}^{{n^2}\times {n^2}}$. 

\begin{equation}
    \left(\bm{A}\otimes \bm{B}\right)\circ \bm{C}
    =\begin{bmatrix}
    {A}_{11}\left[\bm{B}\circ \bm{C}_{11} \right] &\dots&{A}_{1n}\left[\bm{B}\circ \bm{C}_{1n} \right]\\
  \vdots& \vdots & \vdots\\
  {A}_{n1}\left[\bm{B}\circ \bm{C}_{n1} \right] & \dots &  {A}_{nn}\left[\bm{B}\circ \bm{C}_{nn} \right]
    \end{bmatrix}.\label{eq: full Had}
\end{equation}
\noindent 
The computational cost associated with solving the Hadamard product in Eq.~(\ref{eq: full Had}) is $\mathcal{O}(n^{4})$. Unfortunately, unlike sum-factorization~\cite{orszag1979spectral}, it is not possible to reduce the computational cost of Eq.~(\ref{eq: full Had}) by evaluating each direction independently. 

If we add an additional condition, $\bm{A}=\text{diag}(\bm{a}),\:\bm{a}\in\mathbb{R}^{n\times 1}$, then
\begin{equation}
    \left(\bm{A}\otimes \bm{B}\right)\circ \bm{C}
    =\begin{bmatrix}
    a_{1}\left[\bm{B}\circ \bm{C}_{11} \right] &&\mbox{\large 0}\\
  & \ddots & \\
 \mbox{\large 0}  &  &  a_{n}\left[\bm{B}\circ \bm{C}_{nn} \right]
    \end{bmatrix}
 =\text{diag}\begin{pmatrix}
    a_{1}\left[\bm{B}\circ \bm{C}_{11} \right]\\
 \vdots \\
 a_{n}\left[\bm{B}\circ \bm{C}_{nn} \right]
    \end{pmatrix}_{{n^{2}} \times n}
.\label{eq: diag A 2D had tens}
\end{equation}

\noindent The computational cost to evaluate Eq.~(\ref{eq: diag A 2D had tens}) is $\mathcal{O}(n^{3})$. Similarly, if $B=\text{diag}(\bm{b})$ with $\bm{A}$ dense, then $\left(\bm{A}\otimes \bm{B}\right)\circ \bm{C}$ costs $\mathcal{O}(n^3)$ to evaluate by changing the stride through the matrix.

This can be generalized for an arbitrary $d$-sized tensor product, $\left(\bm{A}_1\otimes \bm{A}_2\otimes \dots \otimes \bm{A}_d\right)\circ \bm{C}$, with \\$\bm{A}_1,\: \dots,\: \bm{A}_d\in\mathbb{R}^{n\times n}$ and $\bm{C}\in \mathbb{R}^{{n^d}\times {n^d}}$. If $\bm{A}_i=\text{diag}(\bm{a}_i),\:\bm{a}\in\mathbb{R}^{n\times 1},\:\forall i=1,\dots,d-1$, then,

\begin{equation}
    \left(\bm{A}_1\otimes \bm{A}_2\otimes \dots \otimes \bm{A}_d\right)\circ \bm{C}
    =
   \text{diag} \begin{pmatrix}
    (\bm{a}_1)_1 \dots (\bm{a}_{d-1})_1 \left[\bm{A}_d\circ \bm{C}_{11} \right]\\
 \vdots \\
 (\bm{a}_1)_n \dots (\bm{a}_{d-1})_n \left[\bm{A}_d\circ \bm{C}_{{n^d}{n^d}} \right]
    \end{pmatrix}_{{n^{d}} \times n}
,\label{eq: diag D had tens}
\end{equation}

\noindent and similarly for the other $d-1$ directions through pivoting. Thus, in each of these $d$-cases, the total computational cost is $d n^{d+1}=\mathcal{O}(n^{d+1})\:,\forall n>>d$.

In the context of high-order entropy stable methods, the Hadamard product can always be computed with the diagonal property above, regardless of the basis functions. Consider solving,

\begin{equation}
    \left(\frac{\partial \bm{\chi}\left(\bm{\xi}_v^r\right)}{\partial \xi^\alpha_j} \bm{\Pi}\right)\circ \bm{C},\label{eq: general Had in scheme}
\end{equation}
\noindent 
where $\bm{\chi}$ is some linearly independent, polynomial basis, $\bm{\xi}_v^r$ are a set of nodes that the basis are evaluated on in computational space, $\xi_j$ is a direction that the $\alpha$-th order derivative is applied in, and $\bm{\Pi}$ is the projection operator corresponding to the basis $\bm{\chi}$ such that $\bm{\Pi}\bm{\chi}(\bm{\xi}_v^r)=\bm{I}$. Using Zwanenburg and Nadarajah~\cite[Proposition 2.1 and Corollary 2.2]{zwanenburg_equivalence_2016}, we can always make the substitution $\frac{\partial^\alpha \bm{\ell}(\bm{\xi}_v^r)}{\partial \xi_j^\alpha}=\frac{\partial^\alpha \bm{\chi}(\bm{\xi}_v^r)}{\partial \xi_j^\alpha} \bm{\Pi}$ where $\bm{\ell }$ is the Lagrange basis collocated on the nodes $\bm{\xi}_v^r$--that is 
$\bm{\ell}\left(\bm{\xi}_v^r\right)=\bm{I}$. If we let $\bm{\ell}$ be a tensor product basis, then $\frac{\partial^\alpha \bm{\ell}(\bm{\xi}_v^r)}{\partial \xi_j^\alpha}=\bm{I}(\xi_{i<j})\otimes \frac{d^\alpha \bm{\ell}(\xi_j)}{d\xi_j^\alpha}\otimes \bm{I}(\xi_{i> j})$. Therefore, $\frac{\partial^\alpha \bm{\ell}(\bm{\xi}_v^r)}{\partial \xi_j^\alpha}\circ \bm{C}$ recovers the form of Eq.~(\ref{eq: diag D had tens}), where $\bm{A}_i=\bm{I}\:\forall i=1,\dots,d,\: i \neq j$. Similarly, if we have some weight function that is a non-identity diagonal matrix multiplied to the derivative, then we have,

\begin{equation}
\bm{W}\left(\bm{\xi}_v^r\right)\left(\frac{\partial \bm{\chi}\left(\bm{\xi}_v^r\right)}{\partial \xi^\alpha_j} \bm{\Pi}\right) = \bm{W}\left(\bm{\xi}_v^r\right)\frac{\partial^\alpha \bm{\ell}(\bm{\xi}_v^r)}{\partial \xi_j^\alpha} = \bm{W}\left(\xi_{i<j}\right)\otimes \bm{W}\left(\xi_{j}\right)\frac{d^\alpha \bm{\ell}(\xi_j)}{d\xi_j^\alpha}\otimes \bm{W}(\xi_{i> j}).\label{eq: general weight deriv}
\end{equation}

\noindent Eq.~(\ref{eq: general weight deriv}) closely resembles the stiffness matrix that appears in finite element methods.

\begin{remark}
The same algorithm can be used in evaluating the Hadamard product with the surface integral terms in Chan~\cite[Eq. (58)]{chan2018discretely} by substituting $\alpha=0$ and the facet cubature nodes in Eq.~(\ref{eq: general Had in scheme}).
\end{remark}

This leads to the main finding of this technical note---for Hadamard products arising in spectral methods that involve some $\alpha$-th order derivative of a polynomial function, the matrix assembly and evaluation costs $dn^{d+1}$ flops each.

\begin{thm}
    If the basis function is represented as a tensor product, then the Hadamard product involving some $\alpha$-th order derivative of the basis function costs $dn^{d+1}$ flops.
\end{thm}

\begin{proof}
    Consider we have a basis $\bm{\chi}$ evaluated on a cubature set $\bm{\xi}_v^r$, and we wish to compute $\left(\frac{\partial \bm{\chi}\left(\bm{\xi}_v^r\right)}{\partial \xi^\alpha_j} \bm{\Pi}\right)\circ \bm{C}$. From Zwanenburg and Nadarajah~\cite[Proposition 2.1 and Corollary 2.2]{zwanenburg_equivalence_2016}, we can apply a basis transformation on both $\bm{C}$ and $\frac{\partial \bm{\chi}\left(\bm{\xi}_v^r\right)}{\partial \xi^\alpha_j} \bm{\Pi}$ to a collocated nodal Lagrange set constructed and evaluated on $\bm{\xi}_v^r$. It is important to note that for the basis transformation, we can directly use sum-factorization techniques~\cite{orszag1979spectral} that give an additional $dn^{d+1}$ flops. After the basis transformation, the resulting basis in the Hadamard product is of the form of Eq.~(\ref{eq: general weight deriv}), and the evaluation of the Hadamard product is of the form of Eq.~(\ref{eq: diag D had tens}), and thus is evaluated in $dn^{d+1}$ flops.
\end{proof}

This theorem allows us to solve Hadamard products at $\mathcal{O}\left(n^{d+1}\right)$ for general uncollocated modal schemes in curvilinear coordinates. 

We provide a sample algorithm for implementation in three-dimensions from our in-house PDE solver \texttt{PHiLiP} \enquote{Operators} class. A similar structure is done for the surface Hadamard products where the one-dimensional basis matrices are of size $m\times n,\:m<n$. Let's assume we want to compute $\left(\bm{D}\otimes \bm{W}\otimes \bm{W} \right)\circ \bm{C}_x$, $\left(\bm{W}\otimes\bm{D}\otimes  \bm{W}\right)\circ \bm{C}_y$, and $\left(\bm{W}\otimes  \bm{W}\otimes\bm{D}\right)\circ \bm{C}_z$, where $\bm{W}=\text{diag}\left(\bm{w}\right)$ stores some weights, and $\bm{D}$ is dense. This mimics the Hadamard product to be computed for entropy conserving schemes alike in Chan~\cite[Eq. (58)]{chan2018discretely}. For the tensor product, we let the $x$-direction run fastest, then the $y$-direction, and the $z$-direction runs slowest. We will refer to the first term as the first direction, the second term as the second direction and the third term as the third direction. We evaluate it in three steps. First, we create two vectors of size $\mathbb{R}^{n^{d+1}\times d}$ storing a sparsity pattern: one stores the non-zero row indices and the other stores the non-zero column indices for each of the $d$ directions. From these, we can build an $n^d\times n$-sized matrix, for example both $\bm{D}\otimes \bm{W}\otimes \bm{W}$ and $\bm{C}_x$, storing only the non-zero entries of the general $n^d\times n^d$-sized matrix for each of the $d$ directions. We provide the pesudocode algorithm to build $\bm{D}\otimes \bm{W}\otimes \bm{W}$, $\bm{W}\otimes\bm{D}\otimes  \bm{W}$, and $\bm{W}\otimes  \bm{W}\otimes\bm{D}$. Lastly, after both matrices of size $\mathbb{R}^{n^d\times n}$ for each direction are built, we evaluate the Hadamard product directly.

We generate the sparsity patterns by the algorithm~\ref{alg: sparsity}:

\begin{algorithm}
\caption{Sparsity Pattern Algorithm}\label{alg: sparsity}
\begin{algorithmic}[1]
\State Create rows and columns vectors of size $n^{d+1}\times d$ storing the sparsity pattern.
\State Loop over the $n^4$ indices.
    \For{$i=0;\:i<n;\:i$++}
    \For{$j=0;\:j<n;\:j$++}
    \For{$k=0;\:k<n;\:k$++}
    \For{$l=0;\:l<n;\:l$++}
    

        \State Store the array indices and the non-zero row indices.
        \State array\_index $\gets i * n^3 +j*n^2+k*n+l$
        \State row\_index $\gets i*n^2+j*n +k$
        \State rows[array\_index][0,1,2] $\gets$ row\_index

        \State Store the non-zero column indices through pivoting.
        
        \Comment{Direction 0 (x).}
        \State column\_index\_x $\gets$ $i*n^2+j*n + l$
        \State columns[array\_index][0] $\gets$ column\_index\_x
        
       \Comment{Direction 1 (y).}
        \State column\_index\_y $\gets$ $l*n+ k +i*n^2$
        \State columns[array\_index][1] $\gets$ column\_index\_y 
        
        \Comment{Direction 2 (z).}
        \State column\_index\_z $\gets$ $l*n^2+k+j*n$
        \State columns[array\_index][2] $\gets$ column\_index\_z 
    \EndFor
    \EndFor
    \EndFor
    \EndFor
\end{algorithmic}
\end{algorithm}

Then, using the sparsity patterns, we create the matrices $\bm{D}\otimes\bm{W}\otimes\bm{W}$, $\bm{W}\otimes\bm{D}\otimes  \bm{W}$, and $\bm{W}\otimes  \bm{W}\otimes\bm{D}$ by the algorithm~\ref{alg: basis}:

\begin{algorithm}
\caption{Basis Assembly Algorithm}\label{alg: basis}
\begin{algorithmic}[1]
\State Loop over the $n^4$ indices.
\For{index=0, counter=0; index$<n^{d+1}$; index++, counter++}
\If{counter == n}
\State counter $\gets$ 0
\EndIf

\State Extract the one-dimensional basis values from the sparsity patterns.

\Comment{We use integer division and \% as the mod operator.}

\State x\_row\_index $\gets$ rows[index][0] \% n
\State x\_column\_index $\gets$ columns[index][0] \% n
\State y\_row\_index $\gets$ (rows[index][1] / n) \% n
\State y\_column\_index $\gets$ (columns[index][1] / n) \% n
\State z\_row\_index $\gets$ rows[index][2] / n / n 
\State z\_column\_index $\gets$ columns[index][2] / n / n

\Comment{Basis\_Sparse is an array of matrices of size $d\times \left(n^d\times n\right)$.}

\State Create the matrix storing only non-zero values.

\Comment{Direction 0 (x).}
\State Basis\_Sparse[0][rows[index][0]][counter] $\gets$ basis[x\_row\_index][x\_column\_index] * weights[y\_row\_index] * weights[z\_row\_index];
    
\Comment{Direction 1 (y).}

\State Basis\_Sparse[1][rows[index][1]][counter]
    $\gets$ basis[y\_row\_index][y\_column\_index]
    * weights[x\_row\_index]
    * weights[z\_row\_index];

\Comment{Direction 2 (z).}

\State Basis\_Sparse[2][rows[index][2]][counter]
    $\gets$ basis[z\_row\_index][z\_column\_index]
    * weights[x\_row\_index]
    * weights[y\_row\_index];
    
\EndFor

\end{algorithmic}
\end{algorithm}

\noindent where \enquote{basis} refers to $\bm{D}$ in the given direction, \enquote{weights} refers to $\bm{W}$, and \enquote{Basis\_Sparse} refers to their tensor product storing only the $n^{d+1}$ non-zero values. We can similarly construct $\bm{C}_x$, $\bm{C}_y$ and $\bm{C}_z$ using the sparsity patterns. The third step of evaluating the Hadamard product doesn't require the sparsity patterns since it is the Hadamard product of $\mathbb{R}^{n^d \times n}$ dense matrices.
\section{Results}\label{sec: results}

For numerical verification, we use the open-source Parallel High-order Library for PDEs (\texttt{PHiLiP}, \newline \url{https://github.com/dougshidong/PHiLiP.git})~\cite{shi2021full}, developed at the Computational Aerodynamics Group at McGill University. For the first test, we consider three-dimensions. We let $\left(\bm{C}\right)_{ij}=c_i c_j$, with \\$\bm{c}=\text{rand}([1e^{-8},30])$. We compare the cost of evaluating the three-dimensional Hadamard product \\$\sum_{j=1}^{3}\frac{\partial \bm{\ell}(\bm{\xi}_v^r)}{\partial \xi_j}\circ \bm{C}$ directly, and using our proposed algorithm in Eq.~(\ref{eq: diag D had tens}), for polynomial degrees $n\in[3,15]$.

\begin{figure}[H]
    \centering
    \includegraphics[width=0.8\textwidth]{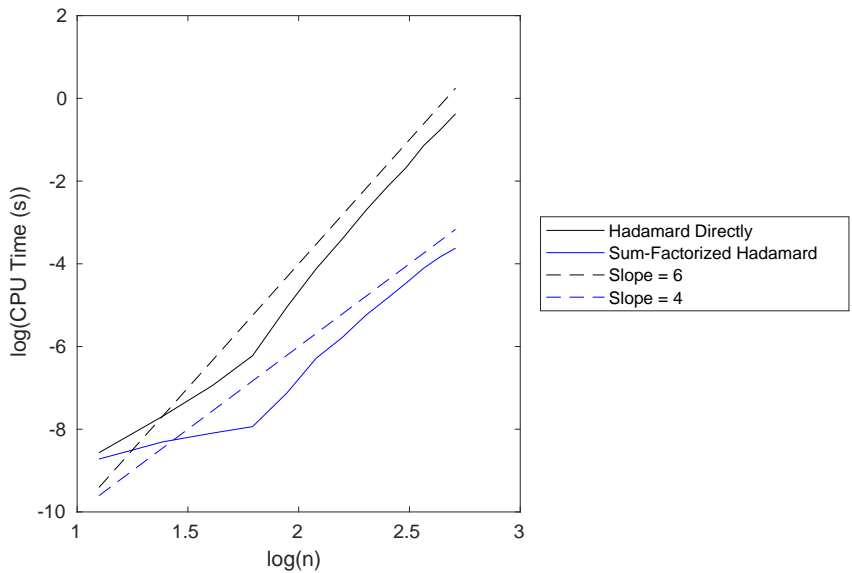}
    \caption{CPU time versus polynomial degree}\label{fig: comp cost}
\end{figure}

In Fig.~\ref{fig: comp cost}, the black solid line corresponds to evaluating the Hadamard directly, the black dashed line corresponds to a slope of 6, the blue solid line corresponds to our proposed method using the tensor-product structure, and the blue dashed line corresponds to a slope of 4. We store the CPU time by running the test on one processor and we record the clock time before the algorithm then subtract the clock time after computing $\sum_{j=1}^{3}\frac{\partial \bm{\ell}(\bm{\xi}_v^r)}{\partial \xi_j}\circ \bm{C}$. In Fig.~\ref{fig: comp cost}, the conventional way of computing a Hadamard product in all three directions in three-dimensions costs $\mathcal{O}(n^{2d})$, whereas our proposed \enquote{sum-factorized} form that exploits the tensor product structure costs $\mathcal{O}(n^{d+1})$.

Next, using our proposed sum-factorized Hadamard product, we wish to compare the performance of the entropy-conserving scheme with the conservative DG scheme using sum-factorization techniques. We solve the three-dimensional inviscid Taylor-Green vortex (TGV) problem on a coarse curvilinear grid using NSFR on uncollocated Gauss-Legendre quadrature nodes. We solve it in six different ways. First, with the conservative DG scheme that does not require a Hadamard product. Second, the conservative DG scheme over-integrated by $2(p+1)$ to resemble exact integration for a cubic polynomial on a curvilinear grid. We consider over-integration because it is another tool used for stabilization~\cite{winters2018comparative} through polynomial dealiazing. Lastly, with our NSFR entropy conserving scheme~\cite{CicchinoNonlinearlyStableFluxReconstruction2021,cicchino2022provably} that requires an uncollocated Hadamard product along with entropy projection techniques~\cite{chan2019skew}. We then, in dashed lines, run the same tests with an FR correction value of $c_{+}$~\cite{vincent_insights_2011} to compare the additional cost of FR versus its DG equivalent. For the test, we perform 10 residual solves sequentially and record the total CPU time for the 10 residual solves.

 \begin{figure}[H]
    \centering
    \includegraphics[width=0.6\textwidth]{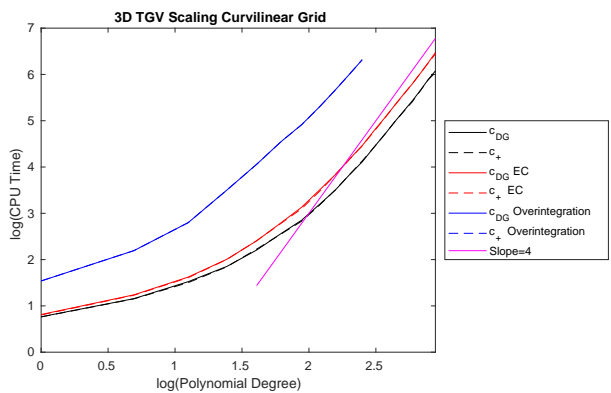}
    \caption{CPU time versus polynomial degree TGV}\label{fig: Philip scaling}\label{fig: tgv scale}
\end{figure}

From Figure~\ref{fig: tgv scale}, all three methods have the solver scale at order $\mathcal{O}\left(p^{d+1}\right)$ in curvilinear coordinates because they exploit sum-factorization~\cite{orszag1979spectral} for the matrix-vector products, and the NSFR-EC scheme uses our proposed sum-factorized Hadamard product evaluation. The blue line representing the over-integrated conservative DG scheme took the most amount of time, and the cut-off at extremely high polynomial orders, $p>20$, was due to memory issues with storing the additional quadrature nodes. The conservative DG scheme took the least amount of time, but the entropy conserving scheme involving the Hadamard product with the two-point flux had a comparable CPU time thanks to the algorithm presented in Section~\ref{sec: hadamard product}. Also, there was a negligible computational cost difference between all $c_\text{DG}$ versus $c_+$ schemes since the mass matrix inverse was approximated in a weight-adjusted form. From Fig.~\ref{fig: tgv scale}, it appears that using the algorithm in Sec.~\ref{sec: hadamard product}, the entropy conserving scheme's cost is more comparable to the conservative DG scheme rather than an over-integrated/exactly integrated DG scheme.

To further demonstrate the performance differences between the NSFR-EC-DG scheme using the \enquote{sum-factorized} Hadamard product evaluations detailed in Sec.~\ref{sec: hadamard product} and the conservative DG scheme in curvilinear coordinates, we run the inviscid TGV on a non-symmetrically warped curvilinear grid and compare the wall clock times. All schemes use an uncollocated, modal Lagrange basis, and are integrated on Gauss-Legendre quadrature nodes. We integrate in time with a 4-$th$ order Runge-Kutta time-stepping scheme with an adaptive Courant-Friedrichs-Lewy value of 0.1 until a final time of $t_f = 14$ s. For NSFR-EC-DG we use Chandrashekar's flux~\cite{chandrashekar2013kinetic} in the volume and surface with Ranocha's pressure fix~\cite{ranocha2022preventing}. For the DG conservative scheme, we use the Roe~\cite{roe1981approximate} surface numerical flux. All of the tests were run on a single node provided by the Digital-Alliance of Canada.

\begin{table}[H]\centering
\caption{TGV Wall Clock Time}
\begin{tabular}{|c|c|c|c|}
\hline
$p$ & Number of Elements & Scheme & Wall Clock (s)\\
&&&\\
\hline
3& $4^3$ &NSFR-EC-DG&718.35\\
\cline{3-4}
& &DG-cons&604.15\\
\cline{3-4}
& &DG-cons-overint&2508.57\\
\cline{2-4}
& $8^3$ &NSFR-EC-DG&5441.31\\
\cline{3-4}
& &DG-cons&6040.61\\
\cline{3-4}
& &DG-cons-overint&23280.30\\
\hline
4& $4^3$&NSFR-EC-DG&1790.95\\
\cline{3-4}
& &DG-cons&1495.61\\
\cline{3-4}
& &DG-cons-overint& 5955.21\\
\cline{2-4}
& $8^3$&NSFR-EC-DG& 9259.31\\
\cline{3-4}
& &DG-cons& 8090.78\\
\cline{3-4}
& &DG-cons-overint&28198.60\\
\hline
5&  $4^3$&NSFR-EC-DG& 1617.21\\
\cline{3-4}
& &DG-cons& crashed\\
\cline{3-4}
& &DG-cons-overint&11055.10\\
\cline{2-4}
&  $8^3$&NSFR-EC-DG& 12050.60\\
\cline{3-4}
& &DG-cons&crashed\\
\cline{3-4}
& &DG-cons-overint&33741.20\\
\hline
\end{tabular}\label{tab: wall clock table}
\end{table}

From Table~\ref{tab: wall clock table}, the NSFR-EC-DG scheme took about a 11\% longer run time as compared to conservative DG. 
This small percentage difference demonstrates how the algorithm in Sec.~\ref{sec: hadamard product} has drastically reduced the computational cost of computing a two-point flux, since we are not required to do twice nor squared the work. The $p=5$ DG conservative scheme diverged at $t=9.06$ s with a wall clock time of 1369.93 s on the $4^3$ mesh, and at $t=6.70$ s with a wall clock of 4473.26 s on the $8^3$ mesh. This further demonstrates the advantage of the NSFR-EC scheme since it has provable guaranteed nonlinear stability for a reasonable computational cost trade-off. The over-integrated scheme took on average 396\% longer run-time than the NSFR-EC-DG scheme. From both Fig.~\ref{fig: Philip scaling} and Table~\ref{tab: wall clock table}, it is clear that with the proposed sum-factorized Hadamard product, entropy conserving and stable methods are computationally competitive with classical DG schemes.
\section{Conclusion}

We derived and demonstrated a \enquote{sum-factorized} technique to build and compute Hadamard products at $\mathcal{O}\left(n^{d+1}\right)$. With the fast evaluations, the computational cost of entropy conserving and stable schemes becomes computationally competitive with the classical conservative modal discontinuous Galerkin method in general three-dimensional curvilinear coordinates.

\bibliographystyle{model1-num-names}
\biboptions{sort&compress}
\bibliography{bibliographie}

\begin{thebibliography}{34}
\expandafter\ifx\csname natexlab\endcsname\relax\def\natexlab#1{#1}\fi
\providecommand{\url}[1]{\texttt{#1}}
\providecommand{\href}[2]{#2}
\providecommand{\path}[1]{#1}
\providecommand{\DOIprefix}{doi:}
\providecommand{\ArXivprefix}{arXiv:}
\providecommand{\URLprefix}{URL: }
\providecommand{\Pubmedprefix}{pmid:}
\providecommand{\doi}[1]{\href{http://dx.doi.org/#1}{\path{#1}}}
\providecommand{\Pubmed}[1]{\href{pmid:#1}{\path{#1}}}
\providecommand{\bibinfo}[2]{#2}
\ifx\xfnm\relax \def\xfnm[#1]{\unskip,\space#1}\fi
\bibitem[{Orszag(1979)}]{orszag1979spectral}
\bibinfo{author}{S.~A. Orszag},
\newblock \bibinfo{title}{Spectral methods for problems in complex geometrics},
\newblock in: \bibinfo{booktitle}{Numerical methods for partial differential
  equations}, \bibinfo{publisher}{Elsevier}, \bibinfo{year}{1979}, pp.
  \bibinfo{pages}{273--305}.
\bibitem[{Tadmor(1984)}]{tadmor1984skew}
\bibinfo{author}{E.~Tadmor},
\newblock \bibinfo{title}{Skew-self adjoint form for systems of conservation
  laws},
\newblock \bibinfo{journal}{Journal of Mathematical Analysis and Applications}
  \bibinfo{volume}{103} (\bibinfo{year}{1984}) \bibinfo{pages}{428--442}.
\bibitem[{Fisher and Carpenter(2013)}]{fisher2013high}
\bibinfo{author}{T.~C. Fisher}, \bibinfo{author}{M.~H. Carpenter},
\newblock \bibinfo{title}{High-order entropy stable finite difference schemes
  for nonlinear conservation laws: {Finite} domains},
\newblock \bibinfo{journal}{Journal of Computational Physics}
  \bibinfo{volume}{252} (\bibinfo{year}{2013}) \bibinfo{pages}{518--557}.
\bibitem[{Fisher et~al.(2013)Fisher, Carpenter, Nordstr{\"o}m, Yamaleev, and
  Swanson}]{fisher2013discretely}
\bibinfo{author}{T.~C. Fisher}, \bibinfo{author}{M.~H. Carpenter},
  \bibinfo{author}{J.~Nordstr{\"o}m}, \bibinfo{author}{N.~K. Yamaleev},
  \bibinfo{author}{C.~Swanson},
\newblock \bibinfo{title}{Discretely conservative finite-difference
  formulations for nonlinear conservation laws in split form: {Theory} and
  boundary conditions},
\newblock \bibinfo{journal}{Journal of Computational Physics}
  \bibinfo{volume}{234} (\bibinfo{year}{2013}) \bibinfo{pages}{353--375}.
\bibitem[{Fisher(2012)}]{fisher2012high}
\bibinfo{author}{T.~C. Fisher}, \bibinfo{title}{High-order {L2} stable
  multi-domain finite difference method for compressible flows}, Ph.D. thesis,
  Purdue University, \bibinfo{year}{2012}.
\bibitem[{Carpenter et~al.(2014)Carpenter, Fisher, Nielsen, and
  Frankel}]{carpenter2014entropy}
\bibinfo{author}{M.~H. Carpenter}, \bibinfo{author}{T.~C. Fisher},
  \bibinfo{author}{E.~J. Nielsen}, \bibinfo{author}{S.~H. Frankel},
\newblock \bibinfo{title}{Entropy stable spectral collocation schemes for the
  {Navier}--{Stokes} equations: {Discontinuous} interfaces},
\newblock \bibinfo{journal}{SIAM Journal on Scientific Computing}
  \bibinfo{volume}{36} (\bibinfo{year}{2014}) \bibinfo{pages}{B835--B867}.
\bibitem[{Parsani et~al.(2015{\natexlab{a}})Parsani, Carpenter, and
  Nielsen}]{parsani2015entropy}
\bibinfo{author}{M.~Parsani}, \bibinfo{author}{M.~H. Carpenter},
  \bibinfo{author}{E.~J. Nielsen},
\newblock \bibinfo{title}{Entropy stable discontinuous interfaces coupling for
  the three-dimensional compressible {Navier}-{Stokes} equations.},
\newblock \bibinfo{journal}{J. Comput. Phys.} \bibinfo{volume}{290}
  (\bibinfo{year}{2015}{\natexlab{a}}) \bibinfo{pages}{132--138}.
\bibitem[{Parsani et~al.(2015{\natexlab{b}})Parsani, Carpenter, and
  Nielsen}]{parsani2015entropyWall}
\bibinfo{author}{M.~Parsani}, \bibinfo{author}{M.~H. Carpenter},
  \bibinfo{author}{E.~J. Nielsen},
\newblock \bibinfo{title}{Entropy stable wall boundary conditions for the
  three-dimensional compressible {N}avier--{S}tokes equations},
\newblock \bibinfo{journal}{Journal of Computational Physics}
  \bibinfo{volume}{292} (\bibinfo{year}{2015}{\natexlab{b}})
  \bibinfo{pages}{88--113}.
\bibitem[{Parsani et~al.(2016)Parsani, Carpenter, Fisher, and
  Nielsen}]{parsani2016entropy}
\bibinfo{author}{M.~Parsani}, \bibinfo{author}{M.~H. Carpenter},
  \bibinfo{author}{T.~C. Fisher}, \bibinfo{author}{E.~J. Nielsen},
\newblock \bibinfo{title}{Entropy stable staggered grid discontinuous spectral
  collocation methods of any order for the compressible {N}avier--{S}tokes
  equations},
\newblock \bibinfo{journal}{SIAM Journal on Scientific Computing}
  \bibinfo{volume}{38} (\bibinfo{year}{2016}) \bibinfo{pages}{A3129--A3162}.
\bibitem[{Carpenter et~al.(2016)Carpenter, Parsani, Nielsen, and
  Fisher}]{carpenter2016towards}
\bibinfo{author}{M.~H. Carpenter}, \bibinfo{author}{M.~Parsani},
  \bibinfo{author}{E.~J. Nielsen}, \bibinfo{author}{T.~C. Fisher},
\newblock \bibinfo{title}{Towards an entropy stable spectral element framework
  for computational fluid dynamics},
\newblock in: \bibinfo{booktitle}{54th AIAA Aerospace Sciences Meeting},
  \bibinfo{year}{2016}, p. \bibinfo{pages}{1058}.
\bibitem[{Yamaleev and Carpenter(2017)}]{yamaleev2017family}
\bibinfo{author}{N.~K. Yamaleev}, \bibinfo{author}{M.~H. Carpenter},
\newblock \bibinfo{title}{A family of fourth-order entropy stable
  nonoscillatory spectral collocation schemes for the {1-D} {Navier}--{S}tokes
  equations},
\newblock \bibinfo{journal}{Journal of Computational Physics}
  \bibinfo{volume}{331} (\bibinfo{year}{2017}) \bibinfo{pages}{90--107}.
\bibitem[{Crean et~al.(2018)Crean, Hicken, {Del Rey Fern\'andez}, Zingg, and
  Carpenter}]{crean2018entropy}
\bibinfo{author}{J.~Crean}, \bibinfo{author}{J.~E. Hicken},
  \bibinfo{author}{D.~C. {Del Rey Fern\'andez}}, \bibinfo{author}{D.~W. Zingg},
  \bibinfo{author}{M.~H. Carpenter},
\newblock \bibinfo{title}{Entropy-stable summation-by-parts discretization of
  the {Euler} equations on general curved elements},
\newblock \bibinfo{journal}{Journal of Computational Physics}
  \bibinfo{volume}{356} (\bibinfo{year}{2018}) \bibinfo{pages}{410--438}.
\bibitem[{Chen and Shu(2017)}]{chen2017entropy}
\bibinfo{author}{T.~Chen}, \bibinfo{author}{C.-W. Shu},
\newblock \bibinfo{title}{Entropy stable high order discontinuous {Galerkin}
  methods with suitable quadrature rules for hyperbolic conservation laws},
\newblock \bibinfo{journal}{Journal of Computational Physics}
  \bibinfo{volume}{345} (\bibinfo{year}{2017}) \bibinfo{pages}{427--461}.
\bibitem[{{Del Rey Fern\'andez} et~al.(2019){Del Rey Fern\'andez}, Crean,
  Carpenter, and Hicken}]{Crean2019Staggered}
\bibinfo{author}{D.~C. {Del Rey Fern\'andez}}, \bibinfo{author}{J.~Crean},
  \bibinfo{author}{M.~H. Carpenter}, \bibinfo{author}{J.~E. Hicken},
\newblock \bibinfo{title}{Staggered-grid entropy-stable multidimensional
  summation-by-parts discretizations on curvilinear coordinates},
\newblock \bibinfo{journal}{Journal of Computational Physics}
  \bibinfo{volume}{392} (\bibinfo{year}{2019}) \bibinfo{pages}{161--186}.
\bibitem[{Friedrich et~al.(2019)Friedrich, Schn{\"u}cke, Winters, {Del Rey
  Fern{\'a}ndez}, Gassner, and Carpenter}]{FriedrichEntropy2020}
\bibinfo{author}{L.~Friedrich}, \bibinfo{author}{G.~Schn{\"u}cke},
  \bibinfo{author}{A.~R. Winters}, \bibinfo{author}{D.~C. {Del Rey
  Fern{\'a}ndez}}, \bibinfo{author}{G.~J. Gassner}, \bibinfo{author}{M.~H.
  Carpenter},
\newblock \bibinfo{title}{Entropy stable space--time discontinuous {Galerkin}
  schemes with summation-by-parts property for hyperbolic conservation laws},
\newblock \bibinfo{journal}{Journal of Scientific Computing}
  \bibinfo{volume}{80} (\bibinfo{year}{2019}) \bibinfo{pages}{175--222}.
\bibitem[{Gassner(2013)}]{gassner2013skew}
\bibinfo{author}{G.~J. Gassner},
\newblock \bibinfo{title}{A skew-symmetric discontinuous {Galerkin} spectral
  element discretization and its relation to {SBP-SAT} finite difference
  methods},
\newblock \bibinfo{journal}{SIAM Journal on Scientific Computing}
  \bibinfo{volume}{35} (\bibinfo{year}{2013}) \bibinfo{pages}{A1233--A1253}.
\bibitem[{Gassner et~al.(2016)Gassner, Winters, and Kopriva}]{gassner2016split}
\bibinfo{author}{G.~J. Gassner}, \bibinfo{author}{A.~R. Winters},
  \bibinfo{author}{D.~A. Kopriva},
\newblock \bibinfo{title}{Split form nodal discontinuous {Galerkin} schemes
  with summation-by-parts property for the compressible {Euler} equations},
\newblock \bibinfo{journal}{Journal of Computational Physics}
  \bibinfo{volume}{327} (\bibinfo{year}{2016}) \bibinfo{pages}{39--66}.
\bibitem[{Ranocha et~al.(2016)Ranocha, {\"O}ffner, and
  Sonar}]{ranocha2016summation}
\bibinfo{author}{H.~Ranocha}, \bibinfo{author}{P.~{\"O}ffner},
  \bibinfo{author}{T.~Sonar},
\newblock \bibinfo{title}{Summation-by-parts operators for correction procedure
  via reconstruction},
\newblock \bibinfo{journal}{Journal of Computational Physics}
  \bibinfo{volume}{311} (\bibinfo{year}{2016}) \bibinfo{pages}{299--328}.
\bibitem[{Ranocha et~al.(2017)Ranocha, {\"O}ffner, and
  Sonar}]{ranocha2017extended}
\bibinfo{author}{H.~Ranocha}, \bibinfo{author}{P.~{\"O}ffner},
  \bibinfo{author}{T.~Sonar},
\newblock \bibinfo{title}{Extended skew-symmetric form for summation-by-parts
  operators and varying {Jacobians}},
\newblock \bibinfo{journal}{Journal of Computational Physics}
  \bibinfo{volume}{342} (\bibinfo{year}{2017}) \bibinfo{pages}{13--28}.
\bibitem[{Abe et~al.(2018)Abe, Morinaka, Haga, Nonomura, Shibata, and
  Miyaji}]{abe2018stable}
\bibinfo{author}{Y.~Abe}, \bibinfo{author}{I.~Morinaka},
  \bibinfo{author}{T.~Haga}, \bibinfo{author}{T.~Nonomura},
  \bibinfo{author}{H.~Shibata}, \bibinfo{author}{K.~Miyaji},
\newblock \bibinfo{title}{Stable, non-dissipative, and conservative
  flux-reconstruction schemes in split forms},
\newblock \bibinfo{journal}{Journal of Computational Physics}
  \bibinfo{volume}{353} (\bibinfo{year}{2018}) \bibinfo{pages}{193--227}.
\bibitem[{Chan(2018)}]{chan2018discretely}
\bibinfo{author}{J.~Chan},
\newblock \bibinfo{title}{On discretely entropy conservative and entropy stable
  discontinuous {Galerkin} methods},
\newblock \bibinfo{journal}{Journal of Computational Physics}
  \bibinfo{volume}{362} (\bibinfo{year}{2018}) \bibinfo{pages}{346--374}.
\bibitem[{Chan(2019)}]{chan2019skew}
\bibinfo{author}{J.~Chan},
\newblock \bibinfo{title}{Skew-symmetric entropy stable modal discontinuous
  {Galerkin} formulations},
\newblock \bibinfo{journal}{Journal of Scientific Computing}
  \bibinfo{volume}{81} (\bibinfo{year}{2019}) \bibinfo{pages}{459--485}.
\bibitem[{Chan and Wilcox(2019)}]{chan2019discretely}
\bibinfo{author}{J.~Chan}, \bibinfo{author}{L.~C. Wilcox},
\newblock \bibinfo{title}{On discretely entropy stable weight-adjusted
  discontinuous {Galerkin} methods: {Curvilinear} meshes},
\newblock \bibinfo{journal}{Journal of Computational Physics}
  \bibinfo{volume}{378} (\bibinfo{year}{2019}) \bibinfo{pages}{366--393}.
\bibitem[{Chan et~al.(2019)Chan, Del Rey~Fern{\'a}ndez, and
  Carpenter}]{chan2019efficient}
\bibinfo{author}{J.~Chan}, \bibinfo{author}{D.~C. Del Rey~Fern{\'a}ndez},
  \bibinfo{author}{M.~H. Carpenter},
\newblock \bibinfo{title}{Efficient entropy stable gauss collocation methods},
\newblock \bibinfo{journal}{SIAM Journal on Scientific Computing}
  \bibinfo{volume}{41} (\bibinfo{year}{2019}) \bibinfo{pages}{A2938--A2966}.
\bibitem[{Cicchino et~al.(2022{\natexlab{a}})Cicchino, Nadarajah, and {Del Rey
  Fernández}}]{CicchinoNonlinearlyStableFluxReconstruction2021}
\bibinfo{author}{A.~Cicchino}, \bibinfo{author}{S.~Nadarajah},
  \bibinfo{author}{D.~C. {Del Rey Fernández}},
\newblock \bibinfo{title}{Nonlinearly stable flux reconstruction high-order
  methods in split form},
\newblock \bibinfo{journal}{Journal of Computational Physics}
  (\bibinfo{year}{2022}{\natexlab{a}}) \bibinfo{pages}{111094}.
\bibitem[{Cicchino et~al.(2022{\natexlab{b}})Cicchino, Del Rey~Fern{\'a}ndez,
  Nadarajah, Chan, and Carpenter}]{cicchino2022provably}
\bibinfo{author}{A.~Cicchino}, \bibinfo{author}{D.~C. Del Rey~Fern{\'a}ndez},
  \bibinfo{author}{S.~Nadarajah}, \bibinfo{author}{J.~Chan},
  \bibinfo{author}{M.~H. Carpenter},
\newblock \bibinfo{title}{Provably stable flux reconstruction high-order
  methods on curvilinear elements},
\newblock \bibinfo{journal}{Journal of Computational Physics}
  \bibinfo{volume}{463} (\bibinfo{year}{2022}{\natexlab{b}})
  \bibinfo{pages}{111259}.
\bibitem[{Ranocha et~al.(2021)Ranocha, Schlottke-Lakemper, Chan,
  Rueda-Ram{\'\i}rez, Winters, Hindenlang, and Gassner}]{ranocha2021efficient}
\bibinfo{author}{H.~Ranocha}, \bibinfo{author}{M.~Schlottke-Lakemper},
  \bibinfo{author}{J.~Chan}, \bibinfo{author}{A.~M. Rueda-Ram{\'\i}rez},
  \bibinfo{author}{A.~R. Winters}, \bibinfo{author}{F.~Hindenlang},
  \bibinfo{author}{G.~J. Gassner},
\newblock \bibinfo{title}{Efficient implementation of modern entropy stable and
  kinetic energy preserving discontinuous {G}alerkin methods for conservation
  laws},
\newblock \bibinfo{journal}{arXiv preprint arXiv:2112.10517}
  (\bibinfo{year}{2021}).
\bibitem[{Shi-Dong and Nadarajah(2021)}]{shi2021full}
\bibinfo{author}{D.~Shi-Dong}, \bibinfo{author}{S.~Nadarajah},
\newblock \bibinfo{title}{Full-space approach to aerodynamic shape
  optimization},
\newblock \bibinfo{journal}{Computers \& Fluids}  (\bibinfo{year}{2021})
  \bibinfo{pages}{104843}.
\bibitem[{Zwanenburg and Nadarajah(2016)}]{zwanenburg_equivalence_2016}
\bibinfo{author}{P.~Zwanenburg}, \bibinfo{author}{S.~Nadarajah},
\newblock \bibinfo{title}{Equivalence between the {energy} {stable} {flux}
  {reconstruction} and {filtered} {discontinuous} {Galerkin} {schemes}},
\newblock \bibinfo{journal}{Journal of Computational Physics}
  \bibinfo{volume}{306} (\bibinfo{year}{2016}) \bibinfo{pages}{343--369}.
\bibitem[{Winters et~al.(2018)Winters, Moura, Mengaldo, Gassner, Walch, Peiro,
  and Sherwin}]{winters2018comparative}
\bibinfo{author}{A.~R. Winters}, \bibinfo{author}{R.~C. Moura},
  \bibinfo{author}{G.~Mengaldo}, \bibinfo{author}{G.~J. Gassner},
  \bibinfo{author}{S.~Walch}, \bibinfo{author}{J.~Peiro},
  \bibinfo{author}{S.~J. Sherwin},
\newblock \bibinfo{title}{A comparative study on polynomial dealiasing and
  split form discontinuous {G}alerkin schemes for under-resolved turbulence
  computations},
\newblock \bibinfo{journal}{Journal of Computational Physics}
  \bibinfo{volume}{372} (\bibinfo{year}{2018}) \bibinfo{pages}{1--21}.
\bibitem[{Vincent et~al.(2011)Vincent, Castonguay, and
  Jameson}]{vincent_insights_2011}
\bibinfo{author}{P.~Vincent}, \bibinfo{author}{P.~Castonguay},
  \bibinfo{author}{A.~Jameson},
\newblock \bibinfo{title}{Insights from von {Neumann} analysis of
  {high}-{order} flux reconstruction schemes},
\newblock \bibinfo{journal}{Journal of Computational Physics}
  \bibinfo{volume}{230} (\bibinfo{year}{2011}) \bibinfo{pages}{8134--8154}.
\bibitem[{Chandrashekar(2013)}]{chandrashekar2013kinetic}
\bibinfo{author}{P.~Chandrashekar},
\newblock \bibinfo{title}{Kinetic energy preserving and entropy stable finite
  volume schemes for compressible {E}uler and {N}avier-{S}tokes equations},
\newblock \bibinfo{journal}{Communications in Computational Physics}
  \bibinfo{volume}{14} (\bibinfo{year}{2013}) \bibinfo{pages}{1252--1286}.
\bibitem[{Ranocha and Gassner(2022)}]{ranocha2022preventing}
\bibinfo{author}{H.~Ranocha}, \bibinfo{author}{G.~J. Gassner},
\newblock \bibinfo{title}{Preventing pressure oscillations does not fix local
  linear stability issues of entropy-based split-form high-order schemes},
\newblock \bibinfo{journal}{Communications on Applied Mathematics and
  Computation} \bibinfo{volume}{4} (\bibinfo{year}{2022})
  \bibinfo{pages}{880--903}.
\bibitem[{Roe(1981)}]{roe1981approximate}
\bibinfo{author}{P.~L. Roe},
\newblock \bibinfo{title}{Approximate {R}iemann solvers, parameter vectors, and
  difference schemes},
\newblock \bibinfo{journal}{Journal of computational physics}
  \bibinfo{volume}{43} (\bibinfo{year}{1981}) \bibinfo{pages}{357--372}.

\end{thebibliography}

\end{document}